\documentclass{article}
\usepackage{amsmath,amssymb,amsthm,fullpage}
\usepackage{mathrsfs,graphicx}

\newcommand{\chair}{\text{chair}}
\newcommand{\kite}{\text{kite}}

\newcommand{\G}{\mathscr{G}}

\newtheorem{thm}{Theorem}
\newtheorem{lem}[thm]{Lemma}
\newtheorem{cor}[thm]{Corollary}

\theoremstyle{definition}

\newtheorem{exa}{Example}

\title{On realization graphs of degree sequences}
\author{Michael D. Barrus\\
\small Department of Mathematics\\
\small University of Rhode Island\\
\small Kingston, RI 02881\\
\small \tt barrus@uri.edu
}

\begin{document}
\maketitle
\begin{abstract}
Given the degree sequence $d$ of a graph, the realization graph of $d$ is the graph having as its vertices the labeled realizations of $d$, with two vertices adjacent if one realization may be obtained from the other via an edge-switching operation. We describe a connection between Cartesian products in realization graphs and the canonical decomposition of degree sequences described by R.I.~Tyshkevich and others. As applications, we characterize the degree sequences whose realization graphs are triangle-free graphs or hypercubes.

\end{abstract}

\section{Introduction}
Given the degree sequence $d$ of a finite, simple graph, it is usually the case that $d$ has several realizations, drawn from several distinct isomorphism classes. Understanding the structure of some or all of these realizations is a major focus in the study of degree sequences. If we wish to discuss the set of realizations as a whole, however, it is often useful to study an auxiliary graph, the so-called realization graph of $d$.

To describe the realization graph, we need some definitions. In this paper all graphs are simple and have finite, nonempty vertex sets. Given the set $[n]=\{1,\dots,n\}$ and the degree sequence $d=(d_1,\dots,d_n)$ of a simple graph (where we always assume that degree sequences are in descending order), a \emph{realization of $d$} is a graph with vertex set $[n]$ such that the degree of vertex $i$ is $d_i$ for all $i \in [n]$.

Suppose that in a graph $G$ there are four distinct vertices $a,b,c,d$ such that $ab$ and $cd$ are edges of $G$ and $ad$ and $bc$ are not. The \emph{$2$-switch $\{ab,cd\} \rightrightarrows \{ad,bc\}$} is the operation of deleting edges $ab$ and $cd$ and adding edges $ad$ and $bc$ to $G$. (This operation has often also been called a \emph{transfer} or \emph{cycle exchange}.) Note that performing a 2-switch in a graph $G$ results in a graph $G'$ in which every vertex has the same degree as it had before; thus $G'$ is another realization of the degree sequence of $G$.

The \emph{(2-switch) realization graph $\G(d)$} is the graph $(\mathscr{R},\mathscr{E})$, where $\mathscr{R}$ is the set of realizations of $d$, and two vertices $G,G'$ of $\mathscr{R}$ are adjacent if and only if performing some 2-switch changes $G$ into $G'$. Since the operation of undoing a 2-switch is itself a 2-switch, $\G(d)$ may be thought of as an undirected graph. We provide examples of $\G(d)$ for a few specific degree sequences in the next section.

It is unclear from the literature where the realization graph as we have defined it first appeared, though related  notions have appeared in multiple contexts. For instance, the interchange graph of a score sequence, as introduced by Brualdi and Li~\cite{BrualdiLi84}, takes as its vertices the tournaments having a given score sequence; edges join tournaments that differ only on the orientation of a single directed triangle. The paper~\cite{BrualdiLi84} shows that this graph is a regular connected bipartite graph; other results appear in~\cite{ChenEtAl09} and~\cite{Shia87}.

Many authors have studied another interchange graph, also introduced by Brualdi~\cite{Brualdi80}, which has as its vertices the (0,1) matrices having prescribed row and column sums. Here edges join vertices that differ by a simple switch of entries; interpreting each matrix as a biadjacency matrix of a bipartite graph, these switches correspond to 2-switches that preserve the partite sets. The papers~\cite{BrualdiManber83,BrualdiShen02,Jin11,LiZhang94,Qian02,Qian99a,Qian99b,Yuster05,Zhang92} include results on such properties of this interchange graph as its diameter and lengths of its cycles. As Arikati and Peled pointed out in~\cite{ArikatiPeled99}, each of these interchange graphs arises as the realization graph $\G(d)$ for the degree sequence $d$ of a split graph obtained by adding edges to any of the associated bipartite graphs to make one of the partite sets a clique.

The realization graph $\G(d)$ seems to have attracted less attention in its general setting, where $d$ may be the degree sequence of a non-split graph. The best-known result on $\G(d)$ is that it is a connected graph for any degree sequence $d$; this is a consequence of a theorem of Fulkerson, Hoffman, and McAndrew~\cite{FHM65} (Petersen proved the same result for regular degree sequences in~\cite{Petersen1891}; see also Senior~\cite{Senior51}).

A major question of study, proposed in~\cite{Brualdi80} and as yet unresolved, is whether $\G(d)$ always has a Hamiltonian cycle (or is $K_2$). Results on interchange graphs, such as those in~\cite{LiZhang94,Zhang92}, yield partial results for the degree sequences of split graphs. The paper~\cite{ArikatiPeled99} shows that $\G(d)$ is Hamiltonian if $d$ has threshold gap 1.

In this paper we provide a structure theorem for the realization graph $\G(d)$ and comment on another class of degree sequences $d$ for which $\G(d)$ is Hamiltonian. After beginning with some important examples of realization graphs and recalling some definitions in Section 2, we show in Section 3 that a certain structural decomposition of degree sequences due to Tyshkevich~\cite{Tyshkevich80,Tyshkevich00}, called the \emph{canonical decomposition}, allows us to express the corresponding realization graphs as Cartesian products of smaller realization graphs.

This structural result then allows us in Section 4 to characterize the realization graphs that are triangle-free (and equivalently, the realization graphs that are bipartite); these are precisely the realization graphs of degree sequences of pseudo-split matrogenic graphs. We also show that the degree sequences whose realization graphs are hypercubes are precisely the degree sequences of split $P_4$-reducible graphs. (All terms will be defined later.)

Since the Hamiltonicity of a Cartesian product follows from the Hamiltonicity of its factors, the canonical decomposition of a degree sequence may be used as an aid in characterizing $d$ for which $\G(d)$ is Hamiltonian. As an illustration, we conclude Section 4 by using our results to show that all triangle-free realization graphs are Hamiltonian.

Throughout the paper we denote the vertex set of a graph $G$ by $V(G)$. For an integer $n \geq 1$, we use $K_n$, $P_n$, and $C_n$ to denote the complete graph, path, and cycle on $n$ vertices, respectively; the complete bipartite graph with partite sets of sizes $m$ and $n$ is $K_{m,n}$. For a family $\mathcal{F}$ of graphs, a graph $G$ is \emph{$\mathcal{F}$-free} if $G$ contains no element of $\mathcal{F}$ as an induced subgraph.

\section{Preliminaries}\label{sec: preliminaries}
In this section we lay some groundwork, beginning with examples of specific realization graphs that will be important in Section 4.

\begin{exa} \label{exa: 1111}
As shown in~\cite{ArikatiPeled99}, when $d$ is $(1,1,1,1)$ or $(3,2,1,1,1)$, $\G(d)$ is isomorphic to the triangle $K_3$. Figure~\ref{fig: 1111,32111} shows the realizations of both of these sequences; it is easy to verify that from any realization of either of these sequences $d$, either of the other two realizations may be obtained via a single 2-switch.

\begin{figure}
\centering
\includegraphics[width=6cm]{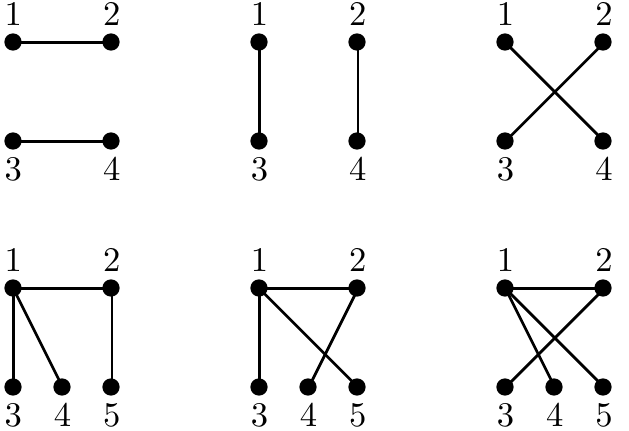}
\caption{Realizations of $(1,1,1,1)$ and $(3,2,1,1,1)$}
\label{fig: 1111,32111}
\end{figure}
\end{exa}

\begin{exa} \label{exa: 22222}
%Example 2: realization graph of C_5
When $d=(2,2,2,2,2)$, the set $\mathscr{R}$ of realizations of $d$ contains 12 graphs, each isomorphic to $C_5$. Each realization allows exactly 5 distinct 2-switches. Performing one of these 2-switches transforms the graph into a cycle that visits the vertices in the same order as before, save that two consecutive vertices on the cycle exchange places. (See Figure~\ref{fig: 22222} for an illustration.) As observed in~\cite{FleischnerEtAl93}, the realization graph $\G(d)$ is then isomorphic to $K_{6,6}$ minus a perfect matching.

\begin{figure}
\centering
\includegraphics[width=5.32cm]{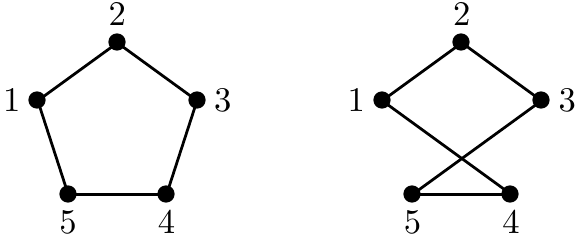}
\caption{Two realizations of $(2,2,2,2,2)$ differing by a single 2-switch}
\label{fig: 22222}
\end{figure}
\end{exa}

\begin{exa} \label{exa: R(net)}
%Example 3: realization graphs of nets
Suppose $d$ is degree sequence $(k,\dots,k,1,\dots,1)$ consisting of $k$ copies of $k$ and $k$ ones (where $k \geq 1$). Every realization of $d$ is a net, where a \emph{($k$-)net} is a graph with vertex set $\{a_1,\dots,a_k, b_1,\dots,b_k\}$ and edge set $\{a_ib_i: i \in [k]\} \cup \{b_ib_j: i,j \in [k]\}$ for some $k\geq 1$. A graph $G$ is a \emph{net-complement} if its complement is a net.

Observe that in any $k$-net, every 2-switch involves edges $a_ib_i,a_jb_j$ and non-edges $a_ib_j,a_jb_i$ for some $i,j \in [k]$, and every pair of distinct elements $i,j$ from $[k]$ yields such a 2-switch. Performing one of these 2-switches has the same effect on a realization as exchanging the names of vertices $a_i$ and $a_j$.

It is easy to see that $\G(d)$ is isomorphic to the \emph{tranposition graph} $T_k$, where the vertices of $T_k$ are permutations of $[k]$ and the edges of $T_k$ join permutations that differ by exactly one transposition. (Thus $T_1 \cong K_1$, $T_2 \cong K_2$, and $T_3 \cong K_{3,3}$.) Transposition graphs are defined in the paper of Chase~\cite{Chase73}, where they are shown to possess Hamiltonian cycles. By an elementary result on the parity of permutations, $T_k$ is bipartite for all $k$.
\end{exa}

Having seen some examples of $\G(d)$, we now turn to the canonical decomposition of a degree sequence as defined by Tyshkevich~\cite{Tyshkevich80,Tyshkevich00}, which will be important for the structure theorem in the next section.

A \emph{split graph} is a graph $G$ for which the vertex set may be partitioned into a clique and an independent set (either of which may be empty). Observe that nets and net-complements are split graphs. If $d$ is the degree sequence of a split graph with a given partition of its vertex set, for convenience we write the degrees of vertices in the clique first, followed by a semicolon and the degrees of vertices in the independent set. Thus the degree sequence of a $k$-net may be written as $(k,\dots,k; 1,\dots,1)$. We refer to such a sequence as a ``splitted'' degree sequence, but we are not always careful to distinguish between splitted and unsplitted degree sequences.

Let $p$ and $q$ be the degree sequences of a split graph $P$ and an arbitrary graph $Q$, respectively. Let $V_1$,$V_2$ be a partition of the vertex set of $P$ into an independent set and a clique. We denote this partitioned graph by $(P,V_1,V_2)$, and we may write the degree sequence of $P$ as the splitted degree sequence $p=(p_2;p_1)$, where $p_2$ consists of the degrees of vertices in $V_2$ and $p_1$ is the list of degrees of vertices in $V_1$. We may assume that $p_2$ and $p_1$ are written with their terms in descending order; it follows that the terms in $p$ then also appear in descending order.

Let $|\ell|$ denote the length of a list $\ell$ of integers. Following Tyshkevich~\cite{Tyshkevich00}, the composition $(p_2;p_1) \circ q$ is the list obtained by concatenating $p$ and $q$, augmenting all terms in $p_2$ by $|q|$ and all terms in $q$ by $|p_2|$, and arranging the sequence into descending order. It is apparent that during this rearrangement of terms, the terms augmented from $p_2$ will appear first, in order, followed immediately by the terms augmented from $q$, in order, followed by the terms of $p_1$.

We may likewise define the composition of a splitted graph and a graph. Informally, given a realization $(P,V_1,V_2)$ of $(p_2;p_1)$ and a realization $Q$ of $q$, the composition $(P,V_1,V_2) \circ Q$ is the graph obtained by taking the disjoint union of $P$ and $Q$ and adding all edges possible between $V_2$ and $V(Q)$.

Since the realizations we deal with in this paper are distinguished not just by isomorphism type but by their edge sets, we now define this operation more precisely. The composition $(P,V_1,V_2) \circ Q$ is the graph with vertex set $[|p|+|q|]$ formed in the following way: We begin by placing an isomorphic copy of $P$ on the vertex set $\{1,\dots,|p_2|\} \cup \{|p_2|+|q|+1,\dots,|p_2|+|q|+|p_1|\}$ where vertex $i$ of $P$ corresponds to vertex $i$, if $1 \leq i \leq |p_2|$, or to vertex $i+|q|$, if $|p_2|+1 \leq i \leq |p_2|+|p_1|$, in $(P,V_1,V_2) \circ Q$. We similarly place an isomorphic copy of $Q$ on the vertex set $\{|p_2|+1,\dots,|p_2|+q\}$ such that vertex $i$ of $Q$ corresponds to vertex $|p_2|+i$ in $(P,V_1,V_2) \circ Q$. We finish the construction of $(P,V_1,V_2) \circ Q$ by adding all edges of the form $uv$, where $u \in \{1,\dots,|p_2|\}$ and $v \in \{|p_2|+1,\dots,|p_2|+|q|\}$. It is easy to see that the sequence $(p_2;p_1) \circ q$ is the degree sequence of $(P,V_1,V_2) \circ Q$.

As an example, the composition $(2,2;1,1)\circ (1,1,1,1)$ is the sequence $(6,6,3,3,3,3,1,1)$. Figure~\ref{fig: composition} depicts the composition $(P,V_1,V_2) \circ Q$, where $P$ is the path on four vertices shown on the left, with $V_1=\{3,4\}$ and $V_2=\{1,2\}$, and $Q$ is the graph isomorphic to $2K_2$.

\begin{figure}
\centering
\includegraphics[width=8cm]{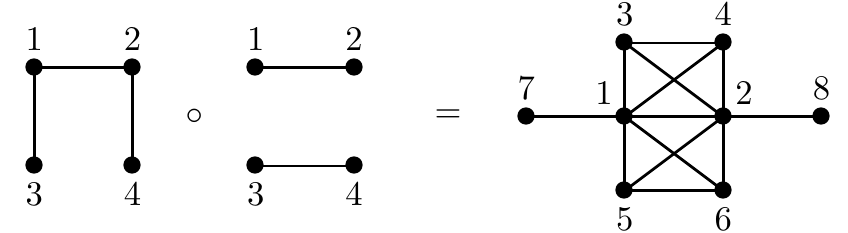}
\caption{The composition $(P,V_1,V_2) \circ Q$}
\label{fig: composition}
\end{figure}

In~\cite{Tyshkevich80,Tyshkevich00} Tyshkevich studied decompositions of degree sequences and graphs with respect to the composition $\circ$ and proved Theorems~\ref{thm: Tysh graphs} and~\ref{thm: Tysh deg seqs} below. Here, a graph or degree sequence is \emph{indecomposable} if it cannot be expressed as the composition of two graphs or degree sequences with strictly smaller sizes. We omit grouping parentheses in the longer compositions (in Tyshkevich's papers analogous compositions of two splitted graphs and of their degree sequences are defined, and with these technical details one verifies that $\circ$ is an associative operation).

\begin{thm}[\cite{Tyshkevich00}] \label{thm: Tysh graphs}
Every graph $F$ can be represented as a composition \begin{equation}\label{eq: canon decomp graphs}
F = (G_1, A_1, B_1) \circ \cdots \circ (G_k, A_k, B_k) \circ F_0
\end{equation} of indecomposable components. Here $(G_i, A_i, B_i)$ are indecomposable splitted graphs and $F_0$ is an indecomposable graph. We call the expression \eqref{eq: canon decomp graphs} the \emph{canonical decomposition} of $F$.

Moreover, graphs $F$ and $F'$ with the canonical decompositions~\eqref{eq: canon decomp graphs} and
\[F' = (G'_1, A'_1, B'_1) \circ \cdots \circ (G'_\ell, A'_\ell, B'_\ell) \circ F'_0\]
are isomorphic if and only if the following conditions hold: \textup{(a)} $F_0 \cong F'_0$; \textup{(b)} $k = \ell$; and \textup{(c)} for all $i \in \{1,\dots,k\}$, there is an isomorphism from $V(G_i)$ to $V(G'_i)$ that sends $A_i$ to $A'_i$ and $B_i$ to $B'_i$.
\end{thm}

\begin{thm}[\cite{Tyshkevich00}]\label{thm: Tysh deg seqs}
Each degree sequence $d$ can be uniquely represented as a composition \begin{equation}\label{eq: canon decomp deg seqs}
d = \alpha_1 \circ \cdots \circ \alpha_k \circ d_0
\end{equation} of indecomposable components. Here the $\alpha_i$ are indecomposable splitted sequences $\alpha_i = (\beta_i; \gamma_i)$, and $d_0$ is an indecomposable graphical sequence. We call the expression \eqref{eq: canon decomp deg seqs} the \emph{canonical decomposition} of $d$.

Moreover, an arbitrary realization $F$ of the sequence $d$ can be represented in the  form~\eqref{eq: canon decomp graphs} above, where $(G_i, A_i, B_i)$ and $F_0$ are realizations of $\alpha_i$ and $d_0$, respectively. All combinations of such realizations yield all realizations of $d$.
\end{thm}

\section{A structure theorem for realization graphs}

In this section we show how the composition operation $\circ$ introduced in the previous section is directly related to Cartesian products in realization graphs.

We recall that the Cartesian product of graphs $G_1=(V_1,E_1)$ and $G_2=(V_2,E_2)$ is the graph $G_1 \Box G_2$ having as its vertex set the Cartesian product $V_1 \times V_2$, where vertices $(u,v)$ and $(w,x)$ are adjacent if and only if either $u=w$ and $vx$ is an edge of $G_2$ or $v=x$ and $uw$ is an edge of $G_1$.

\begin{lem}[{\cite[Lemma 1]{Tyshkevich00}}]\label{lem: unique split ptn}
If $S$ is an indecomposable split graph with more than one vertex, then the partition of $V(S)$ into a clique and an independent set is unique. Moreover, if $S$ has degree sequence $(d_1,\dots,d_n)$ in descending order and $q = \max\{i:d_i \geq i-1\}$, then the $q$ vertices of highest degree comprise the clique in $S$ and the remaining $n-q$ vertices comprise an independent set in the unique splitting partition.
\end{lem}

\begin{lem}[{\cite[Prop.~3.11]{BarrusWest12}}]\label{lem: A4 within component}
Given any graph $G$ and any 2-switch possible in $G$, the four vertices involved in the 2-switch all belong to the same component of the canonical decomposition of $G$.
\end{lem}

\begin{lem}\label{lem: composition}
Let $q$ be the degree sequence of a graph and let $p=(p_2;p_1)$ be the splitted degree sequence of an indecomposable split graph. We have

\[\G(p\circ q) \cong \G(p) \Box \G(q).\]
\end{lem}
\begin{proof}
Since $p$ is indecomposable, either $|p_1|+|p_2|=1$ or there are at least two terms in each of $p_1$ and $p_2$~\cite[Prop.~3.5]{BarrusWest12}.

In the first case $p$ is either $(0;)$ or $(;0)$, creating in $p\circ q$ a term equal to either $n-1$ or $0$, respectively, where $n$ is the number of terms in $p \circ q$. The neighborhood of a vertex with one of these degrees is fixed among all realizations of $p \circ q$; there is thus a bijection between realizations of $p \circ q$ and realizations of $q$, and two realizations in the former are adjacent in $\G(p \circ q)$ if and only if the corresponding realizations of $q$ are adjacent in $\G(q)$. Observe further that $\G(p) \cong K_1$ and $K_1 \Box H \cong H$ for any graph $H$, so $\G(p \circ q) \cong \G(q) \cong \G(p) \Box \G(q)$, as claimed.

Suppose instead that $p_1$ and $p_2$ each contain more than one term, and let $\pi$ be the unsplitted degree sequence with terms in $p$. Since $p$ is indecomposable, by Lemma~\ref{lem: unique split ptn} there is only one partition of the terms of $\pi$ into the degrees of clique vertices and independent set vertices in a realization of $\pi$; this partition is precisely that indicated by $(p_2;p_1)$. Thus for each split indecomposable graph $P$ with more than one vertex, there is a unique partition $A,B$ of the vertices of $P$ such that $(P,A,B)$ is a splitted graph.

As we observed in the last section, if $(P,A,B)$ is a realization of $p$ and $Q$ is a realization of $q$, then the composition $(P,A,B) \circ Q$ is a realization of $p \circ q$, and we define the map $\varphi: V(\G(p)\Box\G(q)) \to V(\G(p \circ q))$ by letting $\varphi((P,Q))$ be $(P,A,B) \circ Q$. In the remainder of the proof we show that $\varphi$ is an isomorphism.

Note that for any realization of $p \circ q$ we may, by Theorem~\ref{thm: Tysh deg seqs}, find a triple $(P,A,B)$ and a graph $Q$ such that the realization is $(P,A,B)\circ Q$ and $P$ is a realization of $p$ and $Q$ is a realization of $q$ (here, the graph $Q$ and the degree sequence $q$ are compositions of all but the left-most canonical components in the decompositions in~\eqref{eq: canon decomp graphs} and~\eqref{eq: canon decomp deg seqs}). Then $\varphi((P,Q))$ is the chosen realization of $p \circ q$, and we see that $\varphi$ is surjective.

Let $(P,Q),(P',Q') \in V(\G(p)\Box\G(q))$. It is clear from the definition of the composition operation $\circ$ that if $P\neq P'$ or if $Q \neq Q'$, then the edge sets of $(P,A,B)\circ Q$ and $(P',A',B')\circ Q'$ differ, where the sets $A,B$ and $A',B'$ are the unique pairs of sets partitioning their respective vertex sets into independent sets and cliques. Thus $\varphi((P,Q)) \neq  \varphi((P',Q'))$, and we see that $\varphi$ is injective.

Before showing that $\varphi$ preserves edges and non-edges, we define some notation and terms. For the remainder of the proof consider two vertices $(P_1,Q_1)$ and $(P_2,Q_2)$ in $V(\G(p)\Box\G(q))$, and let $A_1,B_1$ and $A_2,B_2$ be the unique pairs of sets partitioning the $V(P_1)$ and $V(P_2)$ into independent sets and cliques. Recall that by definition, the composition $(P,A,B)\circ Q$ is constructed by adding edges to disjoint copies of $P$ and $Q$, so each vertex in the composition corresponds to a vertex in either $P$ or $Q$ (where the names of vertices in $(P,A,B)\circ Q$ are ``shifted'' as necessary so that the result is a realization of $p \circ q$). By Lemma~\ref{lem: A4 within component} we can then assign a natural correspondence between 2-switches in the composition $(P,A,B)\circ Q$ and 2-switches in the graphs $P$ and $Q$.

If $\varphi((P_1,Q_1))$ and $\varphi((P_2,Q_2))$ are adjacent vertices in $\G(p \circ q)$, then there is a single 2-switch possible in $(P_1,A_1,B_1)\circ Q_1$ that changes the graph into $(P_2,A_2,B_2)\circ Q_2$. By Lemma~\ref{lem: A4 within component}, we conclude that either $P_1=P_2$ and the corresponding 2-switch in $Q_1$ creates $Q_2$; or $Q_1=Q_2$ and the corresponding 2-switch in $P_1$ creates $P_2$. In either case the vertices $(P_1,Q_1)$ and $(P_2,Q_2)$ are adjacent in $\G(p)\Box\G(q)$.

If instead we begin with the assumption that $(P_1,Q_1)$ and $(P_2,Q_2)$ are adjacent in $\G(p)\Box\G(q)$, then in the pairs $P_1,P_2$ and $Q_1,Q_2$, the graphs in one pair are equal and the graphs in the other pair differ only by a single 2-switch. The corresponding 2-switch in $(P_1,A_1,B_1)\circ Q_1$ results in the graph $(P_2,A_2,B_2)\circ Q_2$ (note that Lemma~\ref{lem: unique split ptn} guarantees that $A_1=A_2$ and $B_1=B_2$). Thus $\varphi((P_1,Q_1))$ and $\varphi((P_2,Q_2))$ are adjacent vertices in $\G(p \circ q)$.
\end{proof}

Since, as noted in~\cite{Tyshkevich00}, the composition operation $\circ$ is associative, we may use Lemma~\ref{lem: composition} and induction to obtain our theorem linking canonical decomposition and realization graphs.

\begin{thm}\label{thm: Cartesian product}
If $d$ is a degree sequence, with \[d = \alpha_1 \circ \cdots \circ \alpha_k \circ d_0\] as its canonical decomposition, then \[\G(d) = \G(\alpha_1) \Box \cdots \Box \G(\alpha_k) \Box \G(d_0).\]
\end{thm}

\section{Applications}
The main results of this section are characterizations of degree sequences $d$ for which $\G(d)$ is a triangle-free graph or a hypercube and a proof that each of these realization graphs is Hamiltonian.

A few definitions are necessary in order to present the results. Recall the definition of transposition graphs given in Example~\ref{exa: R(net)}, and let $K_{6,6}-6K_2$ denote the graph formed by removing the edges of a perfect matching from $K_{6,6}$; this graph is the unique (up to isomorphism) 5-regular bipartite graph on 12 vertices.

Matrogenic graphs were defined by F\"{o}ldes and Hammer in~\cite{FoldesHammer78} as graphs $G$ for which the vertex sets of alternating 4-cycles form the circuits of a matroid with ground set $V(G)$. These authors showed that a graph is matrogenic if and only if it contains no copy of the configuration shown in Figure~\ref{fig: matrogenic}, where dotted segments indicate non-adjacencies. 

\begin{figure}
\centering
\includegraphics[width=2cm]{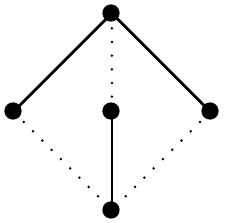}
\caption{The forbidden configuration for matrogenic graphs}
\label{fig: matrogenic}
\end{figure}

A pseudo-split graph is a graph $G$ whose vertices can be organized into disjoint sets $V_1,V_2,V_3$ such that $V_1$ is an independent set, $V_2$ is a clique, no vertex in $V_1$ has a neighbor in $V_3$, each vertex in $V_2$ is adjacent to every vertex in $V_3$, and either $V_3$ is empty or $G[V_3] \cong C_5$. Pseudo-split graphs were shown in~\cite{BlazsikEtAl93} (see also~\cite{MaffrayPreissmann94} and~\cite{MarkossianEtAl96}) to be precisely the graphs containing neither $2K_2$ nor $C_4$ as an induced subgraph.

It is not hard to verify that of the 5-vertex graphs containing the configuration in Figure~\ref{fig: matrogenic}, all induce either $2K_2$ or $C_4$ except for the two connected graphs in Figure~\ref{fig: chair and cochair}, which we call the chair and kite, respectively. Hence a graph is pseudo-split and matrogenic if and only if it is $\{2K_2,C_4,\chair,\kite\}$-free.

\begin{figure}
\centering
\includegraphics[width=6cm]{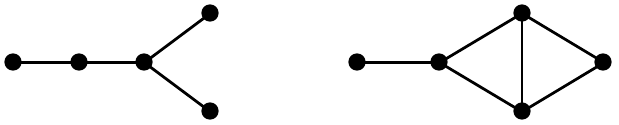}
\caption{The chair and kite graphs}
\label{fig: chair and cochair}
\end{figure}

We now recall some basic facts about realization graphs. Given a degree sequence $d=(d_1,\dots,d_n)$, let $\overline{d}=(n-1-d_n,\dots,n-1-d_1)$. If $d$ is the degree sequence of $G$, then $\overline{d}$ is the degree sequence of the complement $\overline{G}$.

\begin{lem}[\cite{ArikatiPeled99}] \label{lem: complements}
For any degree sequence $d$, $\G(\overline{d}) \cong \G(d)$, and if $d'$ is obtained by appending a 0 to $d$, then $\G(d') \cong \G(d)$.
\end{lem}

By Lemma~\ref{lem: complements} and the examples in Section~\ref{sec: preliminaries}, $\G((2,2,2,2)) \cong \G((3,3,3,2,1)) \cong K_3$, and the realization graph of the degree sequence of any net-complement is bipartite.

\begin{lem} \label{lem: induced}
If $p$ and $q$ are both degree sequences and some realization of $p$ is an induced subgraph of some realization of $q$, then $\G(p)$ is an induced subgraph of $\G(q)$.
\end{lem}
\begin{proof}
Let $P$ and $Q$ be realizations of $p$ and $q$, respectively, for which $P$ is an induced subgraph of $Q$. Note that by replacing the edges of the subgraph $P$ within $Q$ with the edges of any other realization of $p$ we produce another realization of $q$. Furthermore, the realizations of $q$ created in this way satisfy the same adjacency relations in $\G(q)$ that the corresponding realizations of $p$ do in $\G(p)$.
\end{proof}

We note in passing that the relation described in the hypothesis of Lemma~\ref{lem: induced} yields a partial order on the set of degree sequences, which was shown by Chudnovsky and Seymour~\cite{ChudnovskySeymour14} to be a well-quasi-ordering. It follows from Lemma~\ref{lem: induced}, then, that the induced subgraph relation yields a well-quasi-ordering on the class of realization graphs.

\begin{thm}\label{thm: main}
Let $d$ be the degree sequence of a simple graph. The following are equivalent:
\begin{enumerate}
\item[\textup{(a)}] $\G(d)$ is bipartite;
\item[\textup{(b)}] $\G(d)$ is triangle-free;
\item[\textup{(c)}] $\G(d)$ is the Cartesian product of transposition graphs and at most one copy of $K_{6,6} - 6K_2$;
\item[\textup{(d)}] $d$ is the degree sequence of a pseudo-split matrogenic graph.
\end{enumerate}
\end{thm}

\begin{proof}
We show that $\textup{(a)} \Rightarrow \textup{(b)} \Rightarrow \textup{(d)} \Rightarrow \textup{(c)} \Rightarrow \textup{(a)}$.

(a) \emph{implies} (b):  Since $\G(d)$ is bipartite, it contains no odd cycles and hence no triangles.

(b) \emph{implies} (d):  From Example~\ref{exa: 1111} and Lemma~\ref{lem: complements}, we see that the degree sequences of each of $2K_2$, $C_4$, the chair, and the kite have realization graphs containing a triangle. Thus by Lemma~\ref{lem: induced}, since $\G(d)$ is triangle-free, every realization of $d$ is $\{2K_2,C_4,\text{chair},\text{kite}\}$-free. Thus every realization of $d$ is pseudo-split and matrogenic.

(d) \emph{implies} (c): Assume that $d$ is the degree sequence of a pseudo-split matrogenic graph $G$. By results in~\cite{Tyshkevich84} (see also~\cite{MarchioroEtAl84} and Chapter 11 in~\cite{MahadevPeled95}), in the canonical decomposition of any matrogenic graph, the canonical components are each isomorphic to either (1) a single vertex, (2) a net or net-complement, (3) a chordless 5-cycle, or (4) the matching $mK_2$ or its complement, for some $m$. Since $G$ is pseudo-split and hence $\{2K_2,C_4\}$-free, none of the canonical components has form (4), and at most one component has form (3). By Theorem~\ref{thm: Cartesian product} and Examples~\ref{exa: 22222} and~\ref{exa: R(net)}, $\G(d)$ is then the Cartesian product of transposition graphs and at most one copy of $K_{6,6}-6K_2$.

(c) \emph{implies} (a): As observed in Example~\ref{exa: R(net)}, transposition graphs are bipartite. Since the Cartesian product of bipartite graphs is bipartite (as shown in~\cite{Sabidussi57}), if (c) holds then $\G(d)$ is bipartite.
\end{proof}

As a special case of Theorem~\ref{thm: main}, we may characterize those degree sequences $d$ for which $\G(d)$ is a hypercube. Recall that hypercubes are graphs in which each vertex may be identified with a (0,1) vector of fixed length such that two vertices are adjacent in the graph if and only if their vectors differ in exactly one entry. It is an easy exercise to show that the Cartesian product of any number of graphs, all of which are isomorphic to $K_1$ or $K_2$, yields a hypercube.

A graph is \emph{$P_4$-reducible} if each of its vertices belongs to at most induced path on four vertices. As shown in~\cite{GiakoumakisVanherpe97}, a graph $G$ is $P_4$-reducible if and only if $G$ is $\{C_5,P_5,\overline{P_5}, P, \overline{P}, \text{chair}, \text{kite}, \text{3-net}, \overline{\text{3-net}}\}$-free, where $P$ is the graph obtained by attaching a pendant vertex at one vertex of a chordless 4-cycle.

Split graphs, like $P_4$-reducible graphs, have a characterization in terms of forbidden induced subgraphs; F\"{o}ldes and Hammer showed~\cite{FoldesHammer77} that a graph is split if and only if it is $\{2K_2,C_4,C_5\}$-free.

\begin{thm}\label{thm: hypercubes}
Let $d$ be the degree sequence of a graph. The realization graph $\G(d)$ is a hypercube if and only if $d$ is the degree sequence of a split $P_4$-reducible graph.
\end{thm}
\begin{proof}
Comparing the forbidden subgraph characterizations of split graphs and $P_4$-reducible graphs, we see that a graph $G$ is a split $P_4$-reducible graph if and only if it is $\{2K_2,C_4,C_5,\text{chair},\text{kite},\text{3-net},\overline{\text{3-net}}\}$-free, which is true if and only if $G$ is pseudo-split and matrogenic and contains no induced subgraph isomorphic to $C_5$, a 3-net, or the complement of a 3-net.

Suppose that $\G(d)$ is a hypercube. Since $\G(d)$ is bipartite, Theorem~\ref{thm: main} shows that $d$ is the degree sequence of a pseudo-split matrogenic graph. As stated in the proof of Theorem~\ref{thm: main}, this implies that if $G$ is any realization of $d$, then the components of $G$ in the canonical decomposition are each isomorphic to a $K_1$, a net or net-complement, or a chordless 5-cycle. It follows that if $G$ contained an induced subgraph isomorphic to $C_5$, a 3-net, or the complement of a 3-net, then some canonical component of $G$ would be isomorphic to $C_5$ or a $k$-net or the complement of a $k$-net for $k \geq 3$. However, this would cause a contradiction, since Theorem~\ref{thm: Cartesian product} would imply that $\G(d)$ would induce either $K_{6,6}-6K_2$ or some transposition graph $T_k$ for $k \geq 3$; these latter graphs are not induced subgraphs of any hypercube, since they each induce $K_{2,3}$, and in a hypercube each pair of nonadjacent vertices at distance 2 is joined by exactly two paths of length 2. Hence $G$ is pseudo-split and matrogenic and contains no induced $C_5$, 3-net, or the complement of a 3-net, and we conclude that $d$ is the degree sequence of a split $P_4$-reducible graph.

Conversely, suppose that $d$ is the degree sequence of a split $P_4$-reducible graph $G$. Then $G$ is pseudo-split and matrogenic, and by the structural characterization of matrogenic graphs mentioned in the proof of Theorem~\ref{thm: main}, each canonical component of $G$ is a single vertex or a net or net-complement. However, since $G$ is $P_4$-reducible, $G$ also contains no induced 3-net or the complement of a 3-net, so the canonical components of $G$ are all isomorphic to either $K_1$ or $P_4$.  Since $\G((0)) \cong K_1$ and $\G((2,2,1,1)) \cong K_2$, we conclude by Theorem~\ref{thm: Cartesian product} that $\G(d)$ is a Cartesian product of copies of $K_1$ and $K_2$ and hence is a hypercube.
\end{proof}

We now turn to the question of whether realization graphs are Hamiltonian. For the graphs $\G(d)$ in Theorem~\ref{thm: main}, the answer is yes.

\begin{cor}
Every triangle-free realization graph is Hamiltonian.
\end{cor}
\begin{proof}
It is a well-known exercise that the Cartesian product of Hamiltonian graphs is Hamiltonian. Since both $K_{6,6}-6K_2$ and every transposition graph (see~\cite{Chase73}) are Hamiltonian, Theorem~\ref{thm: main} implies that $\G(d)$ is Hamiltonian if it is triangle-free.
\end{proof}

In conclusion, we remark that much is still to be learned about realization graphs in general, both in terms of their structure as graphs and in terms of the degree sequences that produce certain desired properties in a realization graph. We leave it as a question to determine if there are interesting theorems similar to Theorem~\ref{thm: main} for other well-known classes of graphs (planar graphs, perfect graphs, etc.).

\section*{Acknowledgments}

The author wishes to thank the anonymous referees for helpful comments on the presentation of the paper.

\end{document}